\documentclass[12pt]{amsart}
\usepackage{hyperref}
\usepackage{amsmath}
\usepackage{amsthm}
\usepackage{amsfonts}
\usepackage{amssymb}
\usepackage{graphicx}
\DeclareGraphicsExtensions{.eps}

\newtheorem{theorem}{Theorem}[section]
\newtheorem{utv*}{Proposition}
\newtheorem{hyp*}{Conjecture}
\newtheorem{lemma}[theorem]{Lemma}
\newtheorem{corollary}[theorem]{Corollary}

\newtheorem{remark}{Remark}
\newtheorem*{th*}{Theorem}
\newtheorem{prop}{Proposition}

\def\sli{\sum\limits}
\def\ili{\int\limits}

\def\la{\lambda}

\def\R{\mathbb{R}}

\def\ep{\varepsilon}
\def\vf{\varphi}

\newcommand{\ci}[1]{_{ {}_{\scriptstyle #1}}}

\newcommand{\1}{\mathbf{1}}

\newcommand{\C}{\mathbb{C}}

\newcommand{\ve}{\varepsilon}

\def\cyr{\fontencoding{OT2}\fontfamily{wncyr}\selectfont}
\DeclareTextFontCommand{\textcyr}{\cyr}

{\end{list}}

\renewcommand{\Im}{\textup{Im}}

\newcounter{vremennyj}

\begin{document}

		\title[Cauchy independent measures and super-additivity of analytic capacity]{Cauchy independent measures and super-additivity of analytic capacity}

\author{Alexander Reznikov}
\address{Alexander Reznikov, , Michigan State University, East Lansing, Michigan, USA}

\author{Alexander Volberg}
\address{Alexander Volberg, Michigan State University, East Lansing, Michigan, USA}

\thanks{ Alexander Volberg \ was partially supported by the U.S.\ NSF grant DMS-0758552}

\begin{abstract}
We show that, given a family of discs centered at a nice curve, the analytic capacities of arbitrary subsets of these discs add up. However we need that the discs in question would be slightly separated, and it is not clear whether the separation condition is essential or not. We apply this result to study the independence of Cauchy integral operators.
\end{abstract}

\maketitle

\section{Introduction}

We call a finite Borel measure with compact support on the plane {\it Cauchy operator} measure if the Cauchy operator $\mathcal{C}_\mu$ is bounded in $L^2(\mu)$.
We call the collection of measures $\{\mu_j\}_{j=1}^\infty$as above {\it  $C$-Cauchy independent measures} if a) $\|\mathcal{C}_{\mu_j}\|_{\mu_j} \le 1$ and b) for $\mu=\Sigma_j \mu_j$ the following holds $\|\mathcal{C}_\mu\|_\mu \le C<\infty$.

We skip $C$ prefix when it will be clear from context and call such families {\it Cauchy independent}.

Notice that nobody forbids to think that $\mu_j=0$ starting with a certain place. Then we have a finite family of measures. Finite family is always Cauchy independent (but the constant $C$ may grow). So two measures are always Cauchy independent with absolute constant $C$. This is not entirely trivial. We prove it below. But our main interest is in situations, when infinite families are independent.

As always by $\gamma(F)$ we denote the {\it analytic capacity} of $F$.

\subsection{Two Cauchy operator measures are Cauchy independent with absolute constant}
\label{two}

The proof of the following lemma is borrowed from \cite{NToV1}.
\begin{lemma}
\label{proposuma}
Let $\mu$ and $\sigma$ be Borel measures with growth of degree $1$ in $\R^2$ such that $\mathcal{C}_\mu$ is bounded in $L^2(\mu)$ and $\mathcal{C}_\sigma$ is bounded in $L^2(\sigma)$. We assume that their norms are at most $1$.Then, $\mathcal{C}_{\mu+\sigma}$
is bounded in $L^2(\mu+\sigma)$ with norm at most $C_0$, where $C_0$ is an absolute constant.
\end{lemma}

\begin{proof}
The boundedness of $\mathcal{C}_\mu$ in $L^2(\mu)$ implies the boundedness of $\mathcal{C}$ from the space of real measures
$M(\R^{2})$ into $L^{1,\infty}(\mu)$. In other words, the following inequality holds for
any $\nu\in M(\R^{2})$ uniformly on $\ve>0$:
$$\mu\bigl\{x\in\R^{2}:|\mathcal{C}_{\ve}\nu(x)|>\lambda\bigr\}\leq c\,\frac{\|\nu\|}\lambda
\qquad\mbox{for all $\lambda>0$.}$$
For the proof, see Theorem 9.1 of \cite{NTrV1}.
Analogously, the same bound holds with $\mu$ replaced by $\sigma$. As a consequence, we infer that for all $\lambda>0$,
$$(\mu+\sigma)\bigl\{x\in\R^{2}:|\mathcal{C}_{\ve}\nu(x)|>\lambda\bigr\}\leq 2c\,\frac{\|\nu\|}\lambda.$$
That is,  $\mathcal{C}$ is bounded from $M(\R^2)$ into $L^{1,\infty}(\mu+\sigma)$. In particular,
$\mathcal{C}_{\mu+\sigma}$ is of weak type $(1,1)$ with respect to $\mu+\sigma$.
This implies that $\mathcal{C}_{\mu+\sigma}$ is bounded in $L^2(\mu+\sigma)$. For the proof, based
on interpolation, see Theorem 10.1 of \cite{NTrV1} (an alternative argument based on a good lambda
inequality can be also found in Chapter 2 of the book \cite{Tolsa-book}).
\end{proof}

\subsection{Cauchy independence of infinite families of Cauchy operator measures}
\label{inf}

The main result is the following

\begin{theorem}
\label{main2}
Let $\mu=\Sigma_j \mu_j$ be as above, and we assume that measures $\mu_j$ are supported on compacts $E_j$ lying in the discs $D_j$ such that $20D_j$ are disjoint. We also assume that measures $\mu_j$ are {\it extremal} in the sense that $\|\mathcal{C}_{\mu_j}\|_{\mu_j} \le 1$ and $\|\mu_j\|\asymp \gamma(E_j)$.
Let $E=\cup_j E_j$. Then this family is Cauchy independent if and only if for any disc $B$,
\begin{equation}
\label{mainc}
\mu(B) \le C_0 \gamma (B\cap E)\,.
\end{equation}
\end{theorem}
\begin{remark}
We mention that the condition $ \mu(B) \le C_0 \gamma(B\cap\cup_j E_j)$ for all $\mbox{ball} \; B$ is necessary, but not sufficient. We show this in Section \ref{sh} of this paper.  
\end{remark}

First we prove the following independence theorem.

\begin{theorem}
\label{main1}
Let $\mu=\Sigma_j \mu_j$ be as above, and we assume that measures $\mu_j$ are supported on compacts $E_j$ lying in the discs $D_j$ such that $20D_j$ are disjoint. We also assume that measures $\mu_j$ are {\it extremal} in the sense that $\|\mathcal{C}_{\mu_j}\|_{\mu_j} \le 1$ and $\|\mu_j\|\asymp \gamma(E_j)$. Then this family is Cauchy independent if for any disc $B$, $\sum_j \gamma(B\cap E_j) \le C_0 \gamma(B\cap\cup_j E_j)$.
\end{theorem}

To prove these Theorems we will need Section \ref{super}, where a certain situation is studied, where the analytic capacity satisfies ``unnatural" super-additivity condition:
\begin{equation}
\label{su1}
\sum_j \gamma(E_j) \le C_0 \gamma(\cup_j E_j)\,.
\end{equation}
This fact, for a particular case of sets $E_j$, was proved by V. Eiderman \cite{E} using Melnikov--Menger's curvature, See also \cite{NV}.

We will be also using repeatedly  the following result from \cite{NToV1}:

\begin{theorem}
\label{comp}
Suppose $\{D_j\}$ are discs on the plane with $2B_j$ being disjoint. Let $\nu, \sigma$ be two positive measures supported in $ \cup_j D_j$ such that
$c_1\nu(D_j)\le \sigma(D_j) \le c_2 \nu(D_j)$, $0<c_1<c_2<\infty$. Then if $\nu$ is a Cauchy operator measure, then $\sigma$ is also a Cauchy operator measure.
\end{theorem}

\section{Super-additivity of analytic capacity}
\label{super}

We start with the following theorem.
A result close to the theorem below was proved (but not stated) in \cite{NV}. Here we use the approach via Marcinkiewicz function, in \cite{NV} the approach was a bit more complicated. 

We also mention that a version of this theorem was proved by V. Eiderman. In the proof he used the ideas on Menger Curvature. 
\begin{theorem}
\label{superth}
Let $D_j$ be circles with centers on the real line $\R$, such that for some $\la>1$ it is true that $\la D_j \cap \la D_k = \emptyset, \; \; j\not=k$. Let $E_j\subset D_j$ be sets. Then there exists a constant $c=c(\la)$, such that
$$
\gamma(\cup E_j) \geqslant c \sli_j \gamma(E_j).
$$
\end{theorem}
\begin{proof}
It is enough to prove the result for finite families of indices $j$. We first notice that  $\gamma_j := \gamma(E_j) \le diam(E_j) \le 2 r_j$, where $r_j$ is the radius of $D_j$.
Let also $y_j$ be the center of $D_j$.
In each $D_j$ we put a horizontal line segment $L_j$ with center at $y_j$ and with capacity $\frac{1}{100}\gamma_j$. Thus, the length of $L_j$ satisfies $\ell_j \leqslant \frac{1}{20} \gamma_j < \frac{r_j}{10}$.

Next, let $f_j$ be the functions that gives the capacity of $E_j$; let $\vf_j$ be the function that gives the capacity of $L_j$. We write
$$
\vf_j(z)=\ili_{L_j}\frac{\vf_j(x)}{x-z}dx, \; \; \; \; \ili \vf_j(x)dx = \frac{\gamma_j}{100}.
$$
Functions $\vf_j(x)$ have uniform bound $\|\vf_j\|_{\infty} \le A$ by absolute constant. In particular, if $\mathcal{F}$  is any subset of indices $j$ we have
\begin{equation}
\label{img}
|\Im \sum_{j\in \mathcal{F}} \vf_j(z)| \le A \int_{\cup_{j\in \mathcal{F}}L_j}\frac{\Im z}{|t-z|^2}\,dt\le A\,, \forall z\in \mathbb{C}\,.
\end{equation}

Our next goal will be to find a family $\mathcal{F}$ of indices such that the following two assertions hold:
\begin{equation}
\label{sglarge}
\sum_{i\in \mathcal{F}} \gamma_i \ge a_1 \sum_j \gamma_j\,,
\end{equation}

\begin{equation}
\label{fg}
\sum_{i\in \mathcal{F}} |f_i(z)-\vf_i(z)| \le a_2\,,\forall z\in \mathbb{C}\,.
\end{equation}

Let us finish the proof of the theorem, taken these assertions for granted (for a short while). Combining \eqref{img} and \eqref{fg} we get $|\Im \sum_{i\in \mathcal{F}} f_i| \le A_1$. On the other hand, the residue at infinity of $F:=\sum_{i\in \mathcal{F}} f_i$ is $\sum_{i\in \mathcal{F}} \gamma_i$. Then, by \cite{G} we conclude
$$
\gamma(\cup_{i\in \mathcal{F}} E_i) \ge \frac{a}{A_1} \sum_{i\in \mathcal{F}} \gamma_i\,.
$$
 Combine this with \eqref{sglarge}. Then we obtain, that
 $$
 \gamma(\cup_j E_j) \ge \gamma(\cup_{i\in \mathcal{F}} E_i) \ge a_3 \sum_j \gamma_j\,,
$$
and Theorem \ref{superth} would be proved. So we are left to chose the family $\mathcal{F}$ such that \eqref{sglarge}, \eqref{fg} hold.


By the Schwartz lemma in the form we borrow from \cite{G}, we have
\begin{equation}
\label{S}
|f_j(z)-\vf_j(z)| \leqslant \frac{A r_j \gamma_j}{dist(z, E_j\cup L_j)^2}\,,\,\, z\notin D(y_j, \la_0 r_j)\,,
\end{equation}
for a {\bf fixed} $\la_0 >1$.
Denote
$$
Q_i:=D(y_i, \la_0 r_i)\,,\,\,\, g_i:= \sum_{j:\,j\neq i}  \frac{ r_j \gamma_j}{D(Q_j, Q_i)^2}\,,
$$
where $D(Q_i, Q_j):= \text{dist}(Q_i, Q_j) +r_i+r_j$. We can consider function $g$ equal to constant $g_j$ on $Q_j$. Often such object is called a Marcinkiewicz function. What is important is that we can estimate $\sum_i g_i\gamma_i$. In fact,
\begin{align*}
\sum_i g_i\gamma_i =\sum_i\gamma_i \sum_{j:\,j\neq i} \frac{r_j\gamma_j}{D(Q_j, Q_i)^2}= \sum_j r_j \gamma_j  \sum_{i:\,i\neq j} \frac{\gamma_i}{D(Q_i, Q_j)^2} \le \\
 \sum_j r_j \gamma_j  \sum_{i:\,i\neq j} \frac{r_i}{D(Q_i, Q_j)^2} \le A_0 \sum_j r_j \gamma_j r_j^{-1}=A_0\sum_j\gamma_j\,.
 \end{align*}
Now we use Tchebysheff inequality. Denote $I^*:=\{i:\, g_i > 10 A_0\}\,,\, I_*:= \{i:\, g_i \le10 A_0\}$.  We immediately see that
\begin{equation}
\label{I}
\sum_{i\in I_*} \gamma_i \ge \frac{9}{10} \sum_j\gamma_j\,.
\end{equation}

Obviously by \eqref{S}
$$
\forall i\,\forall z\in Q_i \,\, \sum_{j:\, j\neq i} |f_j(z)-\vf_j(z)| \le A g_i\,.
$$
This and the choice of $I_*$ imply that
$$
\forall i\in I_*\,\forall z\in Q_i \,\, \sum_{j:\, j\neq i\,,\, j\in I_*} |f_j(z)-\vf_j(z)| \le A g_i\le 10A_0A\,.
$$
But all functions $|f_j|, |\vf_j|$ are bounded by $1$ everywhere. Therefore, the last inequality implies
\begin{equation}
\label{sumfg}
\sum_{j:\,  j\in I_*} |f_j(z)-\vf_j(z)| \le  10A_0A+2=: A_1 \,\, \forall z\in \cup_{i\in I_*} Q_i\,.
\end{equation}
But function $\sum_{j\in I_*}( f_j-\vf_j)$ is analytic in $\mathcal{C}\setminus \cup_{j\in \mathcal{F}} Q_j$ and vanishes at infinity. Therefore, \eqref{sumfg} implies \eqref{fg} if we put $\mathcal{F}:=I_*$. The assertion \eqref{sglarge} is proved in \eqref{I}. We  are done.

\end{proof}

\begin{corollary}
\label{circle}
By the fact that conformal map of the half-plane onto the unit disc  preserves the analytic capacity up to a constant, and by an obvious observation on dilations, we can see that the same theorem  holds true if centers are on a circle, instead of being on the line.
\end{corollary}

\begin{remark}[Open question]
It is not clear if the theorem is true or not when $\la=1$.
\end{remark}
\section{Beginning the proof of Theorem \ref{main1}}
\label{pr1}

In this section we are proving the following theorem (Theorem \ref{main1}).
\begin{theorem}
Let $E_j$ be sets, and $E=\cup E_j$. Suppose $E_j\subset D_j$, where $D_j$ are discs and $20D_j$ are disjoint. Suppose $\mu_j$ are measures on $E_j$, such that $\mu_j(E_j)\sim \gamma(E_j)$. Denote $\mu = \sli \mu_j$. If for any disc $B$ we have
$$
\gamma(E\cap B)\geqslant c_0 \sli \gamma(E_j \cap B),
$$
then $\mathcal{C}_\mu$ is bounded from $L^2(\mu)$ to itself with norm depending only on $c_0$.
\end{theorem}

Before we begin proving the theorem, we need some construction and notation. First, we define new $L_j$. We fix  $2D_j$ and place a ``cross'' at the center of $D_j$, and $N+1$ crosses that touch $\partial (2D_j)$ (see Figure 1). The choice of $N$ will be independent of $j$. The size of these crosses are such that
$$
\gamma(L_j)\sim H^1(L_j) \sim \gamma(E_j),
$$
where $L_j$ is the union of crosses.
We explain how to chose $N$. In fact, $N$ is big enough so that the following holds.
\begin{prop}
If a disc $B$ intersects $D_j$ and $\C \setminus (10D_j)$ then at least one cross from $L_j$ lies inside $B$.
\end{prop}
\begin{figure}
\begin{center}
\includegraphics[width=0.5\linewidth]{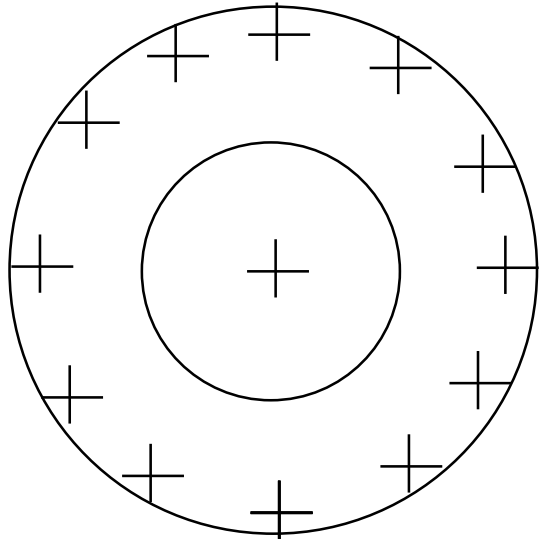} \\
\caption{Definition of $L_j$}
\end{center}
\end{figure}

Next, if a disc $D$ intersects both disc $B$ and $\C\setminus B$ then by $\hat{D}$ we denote the smallest disc with center on $\partial B$ that contains $D$. It is clear that $r(\hat{D})\geqslant r(D)$. Also, $\hat{D}\subset 3D$.

We need following lemmas.
\begin{lemma}
\label{crossB}
For any cross and any disc $B$ it is true with uniform constants that
$$
\gamma(\mbox{cross}\; \cap B) \sim H^1(\mbox{cross}\;\cap B).
$$
\end{lemma}
\begin{lemma}
\label{discinside}
If at least one cross gets inside $B$ then with uniform constants $\gamma(L_j)\sim \gamma(L_j\cap B)$. In particular this is true if $D_j\subset D$, or if $B$ intersects $D_j$ and $\C\setminus\10 D_j$.
\end{lemma}
\begin{proof}
In fact, $\gamma(L_j)\leqslant A \cdot (N+1) \cdot \gamma(\mbox{central cross}) \leqslant A(N+1)\gamma(L_j \cap B)$.
\end{proof}
\begin{lemma}
\label{cutfrominsidediscs}
 With uniform constants
$$
\gamma(\bigcup_{D_j\subset B} L_j)\sim \gamma(\bigcup_{D_j\subset B} L_j \cap B).
$$
\end{lemma}
\begin{proof}
$$
\gamma(\bigcup_{D_j\subset B} L_j) \leqslant A (\gamma(\bigcup_{2D_j\subset B}L_j) + \gamma(\bigcup_{2D_j\not\subset B}L_j)).
$$

The first term is the same as $\gamma(\bigcup_{2D_j\subset B}L_j\cap B)$. For the second, we use that $L_j\cap B\subset \hat{D_j}$, and thus we can apply the Theorem \ref{superth}, or rather Corollary \ref{circle}. In fact, (we always assume $D_j\subset B$),
$$
\gamma(\bigcup_{2D_j\not\subset B} L_j\cap B) \geqslant c\sli_{2D_j\not\subset B} \gamma(L_j\cap B)) \geqslant c_1 \sli_{2D_j\not\subset B} \gamma(L_j) \geqslant c_2 \gamma(\bigcup_{2D_j\not\subset B} L_j),
$$
which finishes the proof.
\end{proof}

\begin{lemma}
\label{cutfromboundarydiscs}
Suppose $B$ intersects more than one $D_j$. Then
$$
\gamma(\bigcup_{\stackrel{D_j\not \subset B}{D_j\cap B\not= \emptyset}} L_j)\sim \gamma(\bigcup_{\stackrel{D_j\not \subset B}{D_j\cap B\not= \emptyset}} L_j \cap B).
$$
\end{lemma}
\begin{proof}
Since $B$ intersects more than one $D_j$, we get that if it does intersect any of these discs then it cannot be contained in $20D_j$. Thus, it contains at least one cross from $L_j$.  Call this cross $C_j$. We again take $\hat{D_j}$ and apply the first theorem to get
$$
\gamma(\bigcup_{\stackrel{D_j\not \subset B}{D_j\cap B\not= \emptyset}} L_j \cap B)\geqslant c \sli_{\stackrel{D_j\not \subset B}{D_j\cap B\not= \emptyset}}\gamma(L_j\cap B) \geqslant
$$
$$
\sli_{\stackrel{D_j\not \subset B}{D_j\cap B\not= \emptyset}} \gamma(\mbox{$C_j$}) \geqslant c_1 \sli_{\stackrel{D_j\not \subset B}{D_j\cap B\not= \emptyset}} \gamma(L_j)\geqslant c_2 \gamma(\bigcup_{\stackrel{D_j\not \subset B}{D_j\cap B\not= \emptyset}}L_j),
$$
which finishes the proof.
\end{proof}

Finally, we need the following notation. Denote
$$
F_j=\begin{cases} E_j, & D_j\subset B \\
\emptyset, & D_j\not\subset B. \end{cases}
$$
$$
F=\cup F_j.
$$
We need to consider two cases. Next two lemmata are devoted to these two separate cases.

 We fix a small $\ep$. The choice of smallness will be clear from what follows.
\begin{lemma}[The first case]\label{firstcase}
Suppose $\gamma(F)\leqslant \ep \gamma(E\cap B)$. Then there exists a constant $c$, such that
$$
\gamma(\bigcup L_j \cap B) \geqslant c\sli \gamma(L_j\cap B).
$$
\end{lemma}

\begin{proof}
For simplicity by $(*)$ we denote the string
$$
D_j\cap B \not= \emptyset \; \mbox{and} \; D_j\not\subset B.
$$

Suppose $B$ intersects only one $2D_j$. Then the $\cup$ and the $\sum$ have only one term, and there is nothing to prove. So, we can assume that $B$ intersects at least two of $2D_j$'s.
Notice also that by this assumption, if $B$ intersects $D_j$ then at least one cross gets inside $B$.
By sub-additivity of $\gamma$
$$
\gamma(\bigcup_{(*)} L_j) \leqslant A\sli_{(*)} \gamma(L_j)\,.
$$
Then using Lemma \ref{discinside} and Corollary \ref{circle} we get

\begin{equation}
\label{zero}
 \sli_{(*)} \gamma(L_j)\leqslant A_1 \sli_{(*)}\gamma(L_j\cap B) \leqslant A_2 \gamma(\bigcup_{(*)}L_j \cap B)\,.
 \end{equation}

On the other hand, by the assumption of the Theorem \ref{main1}
\begin{equation}
\label{star}
\sli_{D_j\subset B}\gamma(L_j) \leqslant c_2\sli_{D_j\subset B}\gamma(E_j) \leqslant c_2 \sli_{\mbox{all}\; j} \gamma(E_j\cap B) \leqslant c_3 \gamma(E\cap B).
\end{equation}

Also,
$$
\gamma(E\cap B) \leqslant A (\gamma(F)+\gamma(\bigcup_{(*)} E_j \cap B)) \leqslant \ep A \gamma(E\cap B) + A\gamma(\bigcup_{(*)} E_j \cap B)).
$$
Thus, if $\ve$ is small enough, we have
\begin{equation}
\label{twostar}
\gamma(E\cap B)\leqslant C \gamma(\bigcup_{(*)} E_j \cap B)).
\end{equation}
Therefore, combining \eqref{star},  \eqref{twostar}, and \eqref{zero}, we obtain
\begin{equation}
\label{3star}
\sli_{D_j\subset B}\gamma(L_j)\leqslant c_4 \gamma(\bigcup_{(*)} E_j \cap B)) \leqslant c_5 \sli_{(*)}\gamma(E_j\cap B) \leqslant c_6 \sli_{(*)}\gamma(L_j) \end{equation}
$$
\leqslant c_7 \sli_{(*)}\gamma(L_j\cap B)\leqslant c_8 \gamma(\bigcup_{(*)} L_j\cap B).
$$

Now combine \eqref{zero} and \eqref{3star} to get
$$
\gamma(\bigcup_{} L_j\cap B) \ge \gamma(\bigcup_{(*)} L_j\cap B) \ge  c \,\sli_{D_j\subset B}\gamma(L_j)  + c \, \sli_{(*)} \gamma(L_j)\,.
$$

Moreover,
$$
\gamma(\bigcup_{} L_j\cap B) \ge \gamma(\bigcup_{\stackrel{D_j\cap B=\emptyset}{L_j\cap B\not=\emptyset}} L_j\cap B).
$$

For these $j$'s we again consider the enlarged discs $\hat{D_j}$. By the Theorem \ref{superth} we get
$$
\gamma(\bigcup_{\stackrel{D_j\cap B=\emptyset}{L_j\cap B\not=\emptyset}} L_j\cap B) \ge \sli_{\stackrel{D_j\cap B=\emptyset}{L_j\cap B\not=\emptyset}}\gamma(L_j\cap B),
$$
which finishes the proof.
\end{proof}

\begin{lemma}[The second case]\label{secondcase}
Suppose that $\gamma(F)\geqslant \ep \gamma(E\cap B)$. Then again
$$
\gamma(\bigcup L_j \cap B) \geqslant c\sli \gamma(L_j\cap B).
$$
\end{lemma}
\begin{proof}
By Theorem \ref{superth} or rather Corollary \ref{circle} we need only to prove
\begin{equation}
\label{inside}
\gamma(\bigcup_{D_j\subset B} L_j\cap B) \geqslant c \sli_{D_j\subset B}\gamma(L_j\cap B).
\end{equation}

We are given that
\begin{equation}\label{reduction}
\gamma(F) \geqslant \ep c \sli \gamma(E_j \cap B) \geqslant c_1 \sli \gamma(F_j).
\end{equation}
In the last inequality we just trow away $j$'s such that $D_j\not\subset B$. Thus, we just forget about all $j$'s for which $D_j\not\subset B$.

By $\nu$ we denote the measure on $F$ that gives $\gamma(F)$. Denote $d\nu_j = \chi_{F_j}d\nu$. Then $\mathcal{C}_{\nu_j}$ is bounded on $L^2(\nu_j)$, and thus
$$
\|\nu_j\| \leqslant c \gamma(F_j) \asymp c_1 \gamma(L_j) \asymp c_2 H^1(L_j) =: c_2 \ell_j.
$$

We call $j$ {\bf good} if $\|\nu_j\| \geqslant \tau \ell_j$. The choice of $\tau$ will be clear from the next steps. We have:
$$
c\sli \gamma(F_j) \leqslant \gamma(\bigcup F_j) =\|\nu\| = \sli \|\nu_j\|
$$
$$
\leqslant c_2 \sli_{j \; \mbox{is good}} \ell_j + \tau \sli \ell_j \leqslant c_2\sli_{j \; \mbox{is good}} \ell_j +c_3\tau\sli \gamma(F_j).
$$
Therefore,
$$
\sli\ci{j \; \mbox{is good}} \ell_j \geqslant c \sli \gamma(F_j)\geqslant c_1 \sli \ell_j.
$$
We call
$$
d\sigma_g := \sli_{j \; \mbox{is good}} \chi_{L_j} dH^1, \,\, d\nu_g := \sli_{j \; \mbox{is good}} \nu_j\,.
$$

Then for good $j$, $\sigma_g(D_j) \sim H^1(L_j) \sim \nu_g(D_j)$.  We use now Theorem \ref{comp}. We get that since $\mathcal{C}{\nu_g}$ is bounded,  $\mathcal{C}_{\sigma_g}$ is also bounded. Therefore,
$$
\gamma(\bigcup_{j: D_j\subset B} L_j) \geqslant \gamma(\bigcup_{j \; \mbox{is good}} L_j) \geqslant c \|\sigma_g\| \geqslant c_1 \sli_{j \; \mbox{is good}} \ell_j \geqslant c_2 \sli_{j: D_j\subset B} \ell_j \geqslant c_3 H^1(\bigcup_{j: D_j \subset B} L_j).
$$

We are done with \eqref{inside} since in Lemma \ref{cutfrominsidediscs} we have proved that
$$
\gamma(\bigcup_{j: D_j\subset B} L_j) \sim \gamma(\bigcup_{j: D_j\subset B} L_j\cap B)
$$
and clearly for every $j$ such that $D_j\subset B$, we have
$$
H^1(L_j)\sim H^1(L_j\cap B)\sim \gamma(L_j\cap B).
$$
\end{proof}

\section{Finishing the proof of Theorem \ref{main1}}
\label{f1}

The main Theorem of \cite{NV} says:

\begin{theorem}
\label{h1}
Let $L$ be a set of positive and finite measure $\nu=H^1|L$. Then $\mathcal{C}_{\nu}$ is bounded if and only if there exists a finite constant $C_0$ such that
 for any disc $B$, $\nu(B\cap L) \le C_0 \gamma(B\cap L)$.
 \end{theorem}

Starting with super-additivity of $\{E_j\}$ (the local one, uniform in arbitrary $B$) we conclude that the same  local super-additivity holds for crosses $\{L_j\}$. Then measure $\nu:=H^1| L$, where $L:=\cup_j L_j$, satisfies this theorem. So the boundedness of Cauchy integral on the union of crosses is obtained. Now we use Theorem \ref{comp} again to conclude the boundedness of $\mathcal{C}_{\mu}$ in $L^2(\mu)$.

\section{Proof of Theorem \ref{main2}}
\label{pr2}
We are going to prove analogs of lemmas \ref{firstcase} and \ref{secondcase}.
\begin{lemma}[The analog of the first case, \ref{firstcase}]
Suppose $\gamma(F)\leqslant \ep \gamma(E\cap B)$. Then there exists a constant $c$, such that
$$
\gamma(\bigcup L_j \cap B) \geqslant c\sli \gamma(L_j\cap B).
$$
\end{lemma}
\begin{proof}
Notice that the only time we used assumptions of the Theorem \ref{main1} in the proof of the Lemma \ref{firstcase} was when we derived \eqref{star}.
We first show that \eqref{star} holds. Let us show that under our new assumption it holds as well.

In fact,
\begin{multline}
\label{star1}
\sli_{D_j\subset B}\gamma(L_j) \leqslant c_2\sli_{D_j\subset B}\gamma(E_j) \leqslant c_3
\sli_{D_j\subset B} \mu_j(D_j) = \\=c_3 \sli_{D_j\subset B} \mu_j(B) \leqslant
c_3 \sli_{\mbox{all}\; j} \mu_j(B) = c_3\mu(B) \leqslant c_4 \gamma(E\cap B).
\end{multline}
The rest of the proof is a word-by-word repetition of the proof of the Lemma \ref{firstcase}.
\end{proof}

\begin{lemma}[The analog of the second case, \ref{secondcase}]
Suppose that $\gamma(F)\geqslant \ep \gamma(E\cap B)$. Then again
$$
\gamma(\bigcup L_j \cap B) \geqslant c\sli \gamma(L_j\cap B).
$$
\end{lemma}
\begin{proof}
Again, the whole proof was based on two facts: \eqref{reduction} and on the further consideration of only those $D_j$'s that are inside $B$.

We start with the estimate:
$$
\gamma(F)\ge \ep \gamma(E\cap B) \ge \ep \mu(B) = \ep \sli_j \mu_j(B) \ge \ep \sli_{D_j \subset B} \mu_j(B) \ge \ep c \sli_{j} \gamma(F_j).
$$
\end{proof}

In the proof of the Lemma \ref{secondcase} we never used the ``super additivity'' except for this place. Therefore, the rest of the proof is again a repetition of the proof of the Lemma \ref{secondcase}.

To finish the proof of the Theorem \ref{main1} we just apply the Theorem \ref{h1} as before.

\section{``Sharpness'' of the Theorem \ref{main2}}
\label{sh}

In this section we show that the condition
\begin{equation}
\label{mug}
\mu(B)\leqslant C \gamma(B\cap E), \; \; \; \forall \; \mbox{ball} \; B
\end{equation}
alone is not enough for the boundedness of $\mathcal{C}_\mu$. Notice that it seems to be the main assumption  \eqref{mainc} of Theorem \ref{main2}. However this assumption alone is not enough for the boundedness of $\mathcal{C}_\mu$, additional conditions on the structure of $\mu$ that seem reasonable are stated in  Theorem \ref{main2}.  In this theorem $\mu$ satisfies \eqref{mug} of course, but in addition $\mu$ consists of  countably many ``separated" pieces, each of which gives a bounded Cauchy operator. Then \eqref{mug} becomes not only necessary, but also a sufficient condition for the boundedness. Without the separation or ``something like that", it is not sufficient.

Let us explain the counterexample shown to us by Xavier Tolsa.
First, we take a square $Q_0 = [0,1]\times [0,1]$. We consider the dyadic sub-squares. Thus, for a natural number $k$ we have $4^k$ sub-squares with sidelength $2^{-k}$. We denote this family by $\mathcal{D}_k = \{Q_k^1, \ldots, Q_k^{4^k}\}$. In every $Q_k^n$ we put the famous Garnett $\frac14$ set. Of course, we need to shrink it and fit into $Q_k^n$. Let us call these sets $G_k^n$. On each $G_k^n$ we define the $1$-Hausdorff measure $d\mu_k^n=\chi_{G_k^n}dH^1$. Then $\|d\mu_{k}^n\| \sim 2^{-k}$. We now denote
$$
\mu = \sli 4^{-k} \mu_k^n.
$$
We first notice that
$$
\|\mu\| \leqslant \sli_k 4^{-k} \sli_{n=1}^{4^k} 2^{-k} \leqslant c.
$$
 Obviously the operator $\mathcal{C}_\mu$ cannot be bounded in $L^2(\mu)$, because otherwise it would be bounded in $L^2(\mu_0)$, where $\mu_0$ is $H^1$ on the initial Garnett set in $Q_0$.

 Notice also that support of $\mu$ is the whole square $Q_0$.
Hence, we need to show that for any ball $B$ it is true that
$$
\mu(B)\leqslant C \gamma(B\cap Q_0).
$$
In fact, let us show this for any square $Q$  instead of $B$ (clearly, it does not matter). Let us first show it for a dyadic square $Q_k^n$. We have
$$
\mu(Q_k^n) \leqslant \sli_{\ell \geqslant k} 4^{-\ell} \sli_{m\colon Q_\ell^m \subset Q_k^n} \mu_\ell^m (Q_\ell^m).
$$
Let us calculate, how many terms we have in the last summation. For $\ell=k$ we have only one term. For $\ell=k+1$ there are exactly $4$ cubes of generation $\ell$ that are in $Q_k^n$. Similarly, on generation $\ell$ there are $4^{\ell-k}$ such cubes. Therefore, we continue the estimate:
$$
\mu(Q_k^n) \leqslant \sli_{\ell\geqslant k} 4^{-\ell}\cdot 4^{\ell-k} \cdot 2^{-\ell} \leqslant c 2^{-k}  \sim \gamma(Q_k^n).
$$
So, our estimate is proved for all dyadic cubes. Given a general cube $Q$ with sidelength $a$, $2^{-k-1}\leqslant a \leqslant 2^{-k}$, we can chose four dyadic cubes of generation $k$, such that $Q$ is inside the union of these cubes. Let us denote them by $Q_k^j$, $j=1,2,3,4$. Then
$$
\mu(Q)\leqslant \sli_{j=1}^4 \mu(Q_k^j) \leqslant 4c 2^{-k} \leqslant 16c \cdot a \sim \gamma(Q).
$$
Thus, our estimate is proved for any cube $Q$, which finishes the example.

\section{Question on superadditivity}
\label{q}

In Theorem \ref{superth} the discs were $\lambda$-separated, where $\lambda>1$. But what if they are just disjoint? Namely, let
$D_j$ be circles with centers on the real line $\R$, such that  it is true that $ D_j \cap  D_k = \emptyset, \; \; j\not=k$. Let $E_j\subset D_j$ be sets. Is it true that then there exists a constant $c=c(\la)$, such that
$$
\gamma(\cup E_j) \geqslant c \sli_j \gamma(E_j)\,?
$$
We cannot either prove or construct a counter-example to this simple claim.


\begin{thebibliography}{AHMTT}




































\bibitem[E]{E}{V. Eiderman}, Personal communication. 2012.

\bibitem[G]{G} {J. Garnett}, Analytic Capacity and Measure. Springer-Verlag. 1972.

\bibitem[NV]{NV}{F. Nazarov, A. Volberg}, {\em Analytic capacity of the portion of continuum and a question of T. Murai}, Proc. of Symposia in Pure Math, v. 79 (2008), pp. 279--292. The volume dedicated to Vladimir Maz'ya's 70th anniversary.

\bibitem[NTrV1]{NTrV1} {F. Nazarov, S. Treil and A. Volberg,}
{\em Weak type estimates and Cotlar inequalities for Calder\'on-Zygmund operators in nonhomogeneous spaces,} Int. Math.
Res. Notices {\bf 9} (1998), 463--487.


\bibitem[NToV1]{NToV1} F.\ Nazarov, X.\ Tolsa and A.\ Volberg, {\em The Riesz transform, rectifiability, and removability for Lipschitz harmonic functions},
Preprint, 2012, pp. 1--14.


















\bibitem[To2]{Tolsa-book} X. Tolsa, {\em  Analytic capacity, the Cauchy transform, and non-homogeneous Calder\'on-Zygmund theory.}  To appear (2012).


\end{thebibliography}
\end{document}